\documentclass[a4paper]{amsart}
\usepackage{amsmath,amsfonts,amssymb,amsthm}
\usepackage{latexsym, xspace, enumerate}
\usepackage[mathscr]{eucal}
\usepackage[all]{xy}
\usepackage{hyperref}
\usepackage{amscd}
\usepackage{mathtools}
\usepackage{vmargin}
\usepackage{rotating}

\newtheorem{theorem}{Theorem}[section]
\newtheorem*{theor}{Main Theorem}

\newtheorem{proposition}[theorem]{Proposition}

\newtheorem{lemma}[theorem]{Lemma}

\newtheorem{corollary}[theorem]{Corollary}
\newtheorem*{corol}{Corollary A}
\newtheorem*{corol2}{Corollary B}

\newtheorem*{KCfields}{Direct finiteness conjecture - fields}
\newtheorem*{KCdrings}{Direct finiteness conjecture - division rings}
\newtheorem*{KCnrings}{Direct finiteness conjecture - noetherian rings}
\theoremstyle{definition}

\newtheorem{question}[theorem]{Question}

\newtheorem{remark}[theorem]{Remark}

\newcommand{\K}{\mathbb{K}}
\newcommand{\R}{\mathbb R}
\newcommand{\N}{\mathbb N}

\newcommand{\F}{\mathbb F}

\newcommand{\Z}{\mathbb Z}

\def\Aut{\mathrm{Aut}}

\def\Cay{\mathrm{Cay}}

\def\Mat{\mathrm{Mat}}

\newcommand{\id}{\text{id}}
\newcommand{\Sp}{\text{Sp}}
\newcommand{\SL}{\text{SL}}
\newcommand{\PSL}{\text{PSL}}
\newcommand{\I}{\mathrm{I}}

\newcommand{\df}{\mathrm{def}}

\makeatletter
\def\Ddots{\mathinner{\mkern1mu\raise\p@
\vbox{\kern7\p@\hbox{.}}\mkern2mu
\raise4\p@\hbox{.}\mkern2mu\raise7\p@\hbox{.}\mkern1mu}}
\title{Groups satisfying Kaplansky's stable finiteness conjecture}
\author{Federico Berlai}
\address{Universit\"{a}t Wien, Fakult\"{a}t f\"{u}r Mathematik, Oskar-Morgenstern-Platz 1, 1090 Wien, Austria.}
\email{federico.berlai@univie.ac.at}
\thanks{The author is supported by the European Research Council (ERC) grant of Prof. Goulnara Arzhantseva, grant agreement no.~259527.}
\date{}

\begin{document}

\begin{abstract}
We prove that every \{finitely generated residually finite\}-by-sofic group satisfies Kaplansky's direct and stable finiteness conjectures with respect
to all noetherian rings. 

We use this result to provide countably many new examples of finitely presented non-LEA groups, for which soficity is still undecided, satisfying these two conjectures.
Deligne's famous example $\widetilde{\Sp_{2n}(\Z)}$ of a non residually finite group is among our examples, along with the families of amalgamated free products $\SL_n(\Z[1/p])\ast_{\F_r}\SL_n(\Z[1/p])$
and HNN extensions $\SL_n(\Z[1/p])\ast_{\F_r}$, where $p>2$ is a prime, $n\geq 3$ and $\F_r$ is a free group of rank $r$, for all $r\geq 2$.
\end{abstract}
\maketitle

\section{Introduction}
\noindent
In \cite[pp. 122-123]{Kap} Irving Kaplansky posed what nowadays is known as \emph{Kaplansky's direct finiteness conjecture}:

\begin{KCfields}
Given a field $\K$ and a group $G$, the group ring $\K[G]$ is directly finite. That is to say, if $x,y\in \K[G]$ are such that $xy=1$, then $yx=1$.
\end{KCfields}
One could look at the matrix rings $\Mat_{n\times n}(\K[G])$ and ask whether or not these rings are directly finite for all $n\in\N$. This is known as \emph{Kaplansky's stable finiteness conjecture}.
Although it might look stronger than the direct finiteness conjecture, they are in fact equivalent \cite{Vir}.
Hence in this paper, for the sake of simplicity, we restrict our arguments to direct finiteness.

Kaplansky himself proved that, given a field $\K$ of characteristic zero and a group $G$, the group ring $\K[G]$ is directly finite.
Since then progress has been done, but the conjecture is still unresolved. In \cite{AOmP}, Ara, O'Meara and Perera proved that $D[G]$ is directly finite whenever $G$ is a residually amenable
group and $D$ is a division ring. Later Elek and Szab\'{o} generalized this result to the wider class of all sofic groups \cite[Corollary 4.7]{ElSz04}.

Sofic groups were introduced in 1999 by Gromov, in an attempt to solve Gottschalk's conjecture in topological dynamics \cite{Gro}. 
The class of sofic groups is far from being completely understood, and it still puzzles the experts.
In particular, it is not yet known if all groups are sofic. 

Soficity is stable under many group-theoretic operations \cite{CaLu,CSC,Pes08}. At the same time, it is still unclear how this notion behaves under taking group extensions:
While it is known that a sofic-by-amenable group is again sofic \cite[Proposition 7.5.14]{CSC}, it is still an open problem whether or not finite-by-sofic, 
free-by-sofic or sofic-by-sofic groups, among others, are again sofic groups. 

Here we use the standard notation of an $\mathcal{A}$-by-$\mathcal{B}$ group to denote a group $G$ with a normal subgroup $N\trianglelefteq G$ such that $N\in\mathcal{A}$ and $G/N\in\mathcal{B}$, where
$\mathcal{A}$ and $\mathcal{B}$ are given classes of groups (e.g. $\mathcal{A}$ being the free groups and $\mathcal{B}$ being the sofic groups, in the case of a free-by-sofic group).

Since there is yet no known example of a group which fails to be sofic, the original Kaplansky's conjecture and the following variant are still open problems:
\begin{KCdrings}
Given a division ring $D$ and a group $G$, the group ring $D[G]$ is directly finite.
\end{KCdrings}

A major recent breakthrough has been made by Virili \cite{Vir}. He proved that any crossed product $N\ast G$ is directly finite whenever $N$ is a left-noetherian ring (respectively: right-noetherian ring)
and $G$ is a sofic group (see Section \ref{prof} for the definition and \cite{Pas,Vir} for more details about crossed products).

Group rings are basic examples of crossed products, hence we state another generalization of the original conjecture:
\begin{KCnrings}
Given a noetherian\footnote{The main concern of this work are groups, hence in what follows we focus on noetherian rings, that is, rings that are both left- and right-noetherian, 
rather than specifying if the ring is left- or right-noetherian. All our statements remain true when one restricts to left-noetherian, or right-noetherian, rings.} 
ring $N$ and a group $G$, the group ring $N[G]$ is directly finite.
\end{KCnrings}

As a consequence of his general result on crossed products, 
Virili deduced that the group ring $N[G]$ of a \{polycyclic-by-finite\}-by-sofic group $G$ is directly finite with respect to all noetherian rings $N$ \cite[Corollary 5.4]{Vir}, 
and that the group ring $D[G]$ 
of a free-by-sofic group $G$ is directly finite with respect to all division rings $D$ \cite[Corollary 5.5]{Vir}. As mentioned above, the interest in these classes of groups arises because they are
not known to be sofic.

The aim of this paper is to prove the following Theorem, which establishes Kaplansky's direct and stable finiteness conjectures for the group ring of many groups that are not known to be sofic.

\begin{theor}\label{Kap RF-by-sofic}
Let $N$ be a noetherian ring and $G$ be a \{finitely generated residually finite\}-by-sofic group. Then $N[G]$ is directly finite and, equivalently, stably finite.
\end{theor}

As a corollary, we partially extend \cite[Corollary 5.5]{Vir} from division rings to noetherian rings. 
\begin{corol}
Let $N$ be a noetherian ring and $G$ be a \{finitely generated free\}-by-sofic group. Then $N[G]$ is directly finite and, equivalently, stably finite.
\end{corol}
Moreover, 
we construct countably many pairwise non-isomorphic finitely presented groups that satisfies both of these conjectures, and that are not known to be sofic.
In particular, these groups are not locally embeddable into amenable groups, also called LEA groups (see Section \ref{sec.4} for definitions).

\begin{corol2}
There exists an infinite family $\{G_i\}_{i\in\N}$ of pairwise non-isomorphic finitely presented non-LEA groups.
These groups are not known to be sofic, and $D[G_i]$ is directly finite and, equivalently, stably finite, with respect to all division rings $D$.
\end{corol2}
The groups described in Corollary B are given by the HNN extensions $\SL_n(\Z[1/p])\ast_{\F_r}$ and by the amalgamated free products
$\SL_n(\Z[1/p])\ast_{\F_r}\SL_n(\Z[1/p])$. Here $n\geq 3$ is an integer, $p>2$ is a prime number and $\F_r$ is a free subgroup of $\SL_n(\Z[1/p])$ of
rank $r$, for all $r\geq2$.
See Corollary~\ref{corKaNi} and Corollary~\ref{cor2} for the precise statements and the proofs.

\medskip

The paper is organized as follows: 
In Section~\ref{sec.spacemarked} we define the space of marked groups, we recall some useful properties and we prove preliminary results that will lead to the proof
of the Main Theorem.
In Section~\ref{prof} we prove the Main Theorem and we give some corollaries. In Section~\ref{sec.4} we apply our result to countably many pairwise non-isomorphic groups,
which are not yet known to be sofic, and for which
we establish Kaplansky's direct finiteness and stable finiteness conjectures. Deligne's famous example $\widetilde{\Sp_{2n}(\Z)}$ of a non residually finite group is among our interests.

\subsubsection*{Acknowledgments}
The author is very grateful to his advisor, Goulnara Arzhantseva, for inspiring questions and suggestions.
He wants to thank Simone Virili for sharing the text of his PhD Thesis, his work \cite{Vir} and for many constructive discussions over the topic. Thanks are also due to Nikolay Nikolov,
for explaining the proof of \cite[Theorem 1]{KaNi}.

\section{The space of marked groups}\label{sec.spacemarked}

In this section, we briefly discuss the space of marked groups. For more details and properties, see~\cite{Ch,ChGu,Gri84}.
We prove the following theorem:

\begin{theorem}\label{RF-by-sofic}
Let $Q$ be a group and $G$ be a finitely generated \{finitely generated residually finite\}-by-$Q$ group.
Then $G$ is the limit, in the space of marked groups, of finite-by-$Q$ groups.  
\end{theorem}
We stress the following particular case:

\begin{corollary}\label{stressed.cor}
Let $G$ be a finitely generated \{finitely generated residually finite\}-by-sofic group.
Then $G$ is the limit, in the space of marked groups, of finite-by-sofic groups. 
\end{corollary}

A \emph{marked group} is a pair $(G,S)$, where $G$ is a finitely generated group and $S$ is a finite sequence of elements that generate $G$. If $\lvert S\rvert=n$ then $(G,S)$ is
called an $n$-marked group. Two $n$-marked groups $(G,(s_1,\dots,s_n))$ and $(G',(s'_1,\dots,s'_n))$ are isomorphic
if the bijection $\varphi(s_i):=s_i'$ extends to an isomorphism of groups $\varphi\colon G\to G'$. In particular, two marked groups with given generating sets of different size are 
never isomorphic as marked groups, although they might be isomorphic as abstract groups.

Let $\mathcal{G}_n$ denote the set of $n$-marked groups, up to isomorphism of marked groups. Then $\mathcal{G}_n$ corresponds bijectively to the set of
normal subgroups of the free group $\F_n$ on $n$ free generators.

Let $(G,S),(G',S')\in\mathcal{G}_n$ and let $N$, $N'$ be the normal subgroups of $\F_n$ such that $\F_n/N\cong G$, $\F_n/N'\cong G'$. Let $B_{\F_n}(r)$ denote the ball of radius $r$ in $\F_n$ centered 
at the identity element. The function
\begin{equation}
v(N,N'):=\sup\{r\in\N\mid N\cap B_{\F_n}(r)=N'\cap B_{\F_n}(r)\}
\end{equation}
defines on $\mathcal{G}_n$ the ultrametric 
\begin{equation}\label{ultrametric}
d\bigl((G,S),(G',S')\bigr):=2^{-v(N,N')}.
\end{equation}
An $n$-marked group $(G,S)$ can be viewed as an $(n+1)$-marked group by adding the trivial element $e_G$ to $S$. This defines an isometric embedding of $\mathcal{G}_n$ into $\mathcal{G}_{n+1}$.

Let $\mathcal{G}:=\bigcup_{n\in\N}\mathcal{G}_n$ be the \emph{space of marked groups}. The ultrametrics of \eqref{ultrametric} can be extended to an ultrametric
$d\colon \mathcal{G}\times\mathcal{G}\to \R_{\geq0}$, and $\bigl(\mathcal{G},d\bigr)$ is a compact totally disconnected ultrametric space~\cite{ChGu}.
The following fact is well known.

\begin{lemma}\label{char.conv}
Let $(G,S)$ and $\{(G_r,S_r)\}_{r\in\N}$ be marked groups in $\mathcal{G}_n$ and let $N, \{N_r\}_{r\in\N}$ be normal subgroups of $\F_n$ such that $\F_n/N\cong G$ and $\F_n/N_r\cong G_r$.
The following are equivalent:
\begin{enumerate}
 \item the sequence $\{(G_r,S_r)\}_{r\in\N}$ converges to the point $(G,S)$ in the space $\bigl(\mathcal{G}_n,d\bigr)$;
 \item for all $x\in N$ (respectively: for all $y\notin N$) there exists $\bar{r}$ such that $x\in N_r$ (respectively: $y\notin N_r$) for $r\geq\bar{r}$;
 \item for all $R\in\N$ there exists $\bar{r}$ such that the Cayley graphs $\Cay(G,S)$ and $\Cay(G_r,S_r)$ have balls of radius $R$ isomorphic as labeled directed graphs, for all $r\geq\bar{r}$.
\end{enumerate}
\end{lemma}

The following easy lemma is used in the proof of Theorem~\ref{RF-by-sofic}.

\begin{lemma}\label{piccololemma}
Let $K\unlhd N\unlhd G$ be a subnormal chain of groups, where $N$ is finitely generated, and $K$ has finite index in $N$.

Then there exists a normal subgroup $\widetilde{K}\unlhd G$, contained in $K$, such that $N/\widetilde{K}$ is finite.
\begin{proof}
Since $N$ is finitely generated, for any given positive integer $d\in\N$, $N$ has only finitely many subgroups of index $d$.
Let $\widetilde{K}$ be the intersection of all subgroups $H\unlhd N$ with $[N:H]=[N:K]$. Then $\widetilde{K}$ is a finite index characteristic subgroup of $N$, and hence it is normal in $G$.

\end{proof}
\end{lemma}
We are now ready to prove Theorem \ref{RF-by-sofic}.

\begin{proof}[\bf Proof of Theorem \ref{RF-by-sofic}]
Let $G$ be a \{finitely generated residually finite\}-by-$Q$ group and suppose $G$ is generated by a finite set $S$.
Let $N\unlhd G$ be a finitely generated residually finite subgroup such that $G/N\cong Q$.

As $N$ is residually finite, there exists a family $\{K_r\}_{r\in\N}$ of finite index normal subgroups of $N$ such that $\bigcap_{r\in\N}K_r=\{e\}$ and $K_{r+1}\subseteq K_r$ for all
$r\in\N$.
As $N$ is finitely generated, we can assume that $K_r\unlhd G$ for all $r\in\N$ by Lemma~\ref{piccololemma}.

Consider the family of finite-by-$Q$ groups $\{G/K_r\}_{r\in\N}$.
For every $r$, let $S_r$ be the image of $S$ in $G/K_r$ under the canonical projection $\lambda_r\colon G\twoheadrightarrow G/K_r$. 
We prove that the sequence $\{(G/K_r,S_r)\}_{r\in\N}$ converges to $(G,S)$ in $\mathcal{G}_{\lvert S\rvert}$ using the second condition of Lemma~\ref{char.conv}.

Let $\pi\colon\F\twoheadrightarrow G$ be the canonical surjective homomorphism from the finitely generated free group $\F$ on $\lvert S\rvert$ free generators, 
and $\pi_r\colon\F\twoheadrightarrow G/K_r$. Then $\pi_r=\lambda_r\circ\pi$:
\begin{equation*}
\xymatrix{\F\ar[rr]^{\pi}\ar[rrd]_{\pi_r}&&G\ar[d]^{\lambda_r}\\
&&G/K_r 
}
\end{equation*}
Set $\Lambda:=\ker\pi$ and $\Lambda_r:=\ker\pi_r$.
By definition, $\Lambda\leq\Lambda_r$ for every $r$, so, to check that the second condition of Lemma~\ref{char.conv} is satisfied,
we have only to prove that for all $y\notin\Lambda$ there exists $\bar{r}$ such that $y\notin\Lambda_r$, for $r\geq \bar{r}$.
Let $y\notin \Lambda$ and $\pi(y)=:g\in G\setminus\{e_G\}$.
Note that if $g\notin N$ then $g\notin K_r$ for all $r$, as $K_r\subseteq N$. Thus, for such $g\notin N$ we have
\begin{equation}
\pi_r(y)=\lambda_r(\pi(y))=\lambda_r(g)\neq e_{G/K_r}\qquad \forall r\in\N,
\end{equation}
that is to say, $y\notin \Lambda_r$ for all $r\in\N$.

If $g\in N\setminus\{e_G\}$, then there exists $\bar{r}$ such that $g\notin K_r$ for all $r\geq \bar{r}$, because $\bigcap_{r\in\N}K_r=\{e\}$ and because this family of normal subgroups
is totally ordered by inclusion.
In particular,
\begin{equation}
\pi_r(y)=\lambda_r(g)\neq e_{G/K_r}\qquad \forall r\geq \bar{r}. 
\end{equation}
This implies that $y\notin\Lambda_r$ for $r\geq\bar{r}$, and the proof is completed.

\end{proof}

\section{Proof of the Main Theorem}\label{prof}
We first show that the assertions of Kaplansky's direct finiteness and stable finiteness conjectures are preserved under taking limits in the space of marked groups.
\begin{proposition}\label{limit}
Let $N$ be a noetherian ring and $(G,S)$ be the limit, in the space of marked groups, of the sequence $\{(G_r,S_r)\}_{r\in\N}$.
If $N[G_r]$ is directly finite (respectively: stably finite) for all $r\in\N$ then so is $N[G]$.

\begin{proof}
Suppose first that $N[G_r]$ is directly finite for all $r\in\N$, and consider two non-trivial elements $x=\sum_{g\in G}k_gg$ and $y=\sum_{g\in G}h_gg\in N[G]$ such that $xy=1$. 
We want to prove that $yx=1$.

Let $yx=\sum_{g\in G}l_gg$ and consider 
$$m:=\max\{\lVert g\rVert_S\mid k_g\neq 0\text{ or }h_g\neq 0\text{ or }l_g\neq 0\},$$
where $\lVert -\rVert_S$ denotes the norm induced by the word metric on $G$ given by the finite generating set $S$. That is to say:
$$\lVert g\rVert_S:=\min\{k\mid g=s_1\dots s_k,\quad s_i\in S\cup S^{-1}\}.$$
Since the sequence $\{(G_r,S_r)\}_{r\in\N}$ converges to $(G,S)$, by Lemma~\ref{char.conv} there exists $\bar{r}$ such that, for all $r\geq \bar{r}$, $\Cay(G,S)$ and $\Cay(G_r,S_r)$ have
balls of radius $m$ isomorphic as labeled directed graphs. 
The group-coefficients of $x$ and $y$ have the same partial multiplication in $G$ as in $G_r$, and moreover $yx=1$ in $N[G_r]$. It thus follows that $yx=1$ in $N[G]$.

The same arguments work when $N[G_r]$ is stably finite for all $r\in\N$.

\end{proof}
\end{proposition}

\begin{proof}[\bf Proof of the Main Theorem]
First we suppose that the group in question is finitely generated, so 
let $G$ be a finitely generated \{finitely generated residually finite\}-by-sofic group. By Corollary~\ref{stressed.cor}, $G$ is the limit of finite-by-sofic groups, which 
satisfy Kaplansky's direct finiteness conjecture with respect to all noetherian rings \cite[Corollary 5.4]{Vir}. Hence, by Proposition~\ref{limit}, $G$ satisfies the conjecture
with respect to all noetherian rings.

If the group $G$ is not finitely generated, consider a finitely generated residually finite normal subgroup $K$ such that $G/K$ is sofic.
Then $G$ is the directed union of its finitely generated subgroups containing such $K$:
\begin{equation}\label{dir.union}
G=\bigcup\{H\mid K\leq H\leq G\text{ and }H\text{ is finitely generated}\}. 
\end{equation}
These are finitely generated \{finitely generated residually finite\}-by-sofic groups, and hence satisfy
Kaplansky's direct finiteness conjecture by the first part of the proof. 

Fix a noetherian ring $N$ and consider two elements $x,y\in N[G]$ such that $xy=1$.
The group-coefficients of $x$, $y$ and $yx$, being non-trivial only finitely many times, sit in some finitely generated subgroup $H\leq G$ appearing in the directed union in \eqref{dir.union}.

The group ring $N[H]$ is directly finite by the first part of this proof, so $yx=1$ in $N[H]$.
This implies that $yx=1$ in $N[G]$.
Thus, $G$ satisfies the conjecture.
\end{proof}

There is a variant of Lemma~\ref{piccololemma}, in the case when the group $N/K$ is solvable: 
Assume we are given a subnormal chain $K\unlhd N\unlhd G$ and suppose that $N/K$ is solvable, then there exists a normal subgroup $\widetilde{K}\unlhd G$, contained in $K$, such that $N/\widetilde{K}$ is solvable as well. 

To the author's knowledge, the following are open questions:
\begin{question}\label{question1}
Let $N$ be a noetherian ring and $G$ be a solvable-by-sofic group. Is $N[G]$ directly finite? 
\end{question}
\begin{question}
Let $N$ be a noetherian ring and $G$ be a solvable group. Does there exist a noetherian ring $N'$ such that $N[G]$ embeds into $N'$?
\end{question}
An affirmative answer to the latter question implies an affirmative answer to Question \ref{question1}.
If Question \ref{question1} has an affirmative answer, then our argument as in
the proof of Theorem~\ref{RF-by-sofic} implies that a finitely generated \{residually solvable\}-by-sofic group $G$ is 
the limit in $\mathcal{G}$ of solvable-by-sofic groups, and hence that a \{residually solvable\}-by-sofic group is directly finite.

\medskip

We recall now the definition and basic facts on crossed products. They are useful in the applications of our Main Theorem.

Given a ring $R$ and a group $G$, a \emph{crossed product} $R\ast G$ of $G$ over $R$ is a ring constructed as follows. 
Assign uniquely to every $g\in G$ a symbol $\bar{g}$, and let $\bar{G}$ be the collection of these symbols.
As a set, 
\begin{equation*}
 R\ast G:=\Bigl\{\sum_{g\in G}r_g \bar{g}\mid g\in G, r_g\in R\text{ is almost always }0 \Bigr\}.
\end{equation*}
The sum is defined component-wise,
\begin{equation*}
\Bigl(\sum_{g\in G}r_g\bar{g}\Bigr)+ \Bigl(\sum_{g\in G}s_g\bar{g}\Bigr):=\sum_{g\in G}(r_g+s_g)\bar{g}.
\end{equation*}
The product in $R\ast G$ is specified in terms of two maps 
\begin{equation*}
 \tau\colon G\times G\to U(R),\qquad \sigma\colon G\to \Aut(R),
\end{equation*}
where $U(R)$ is the group of units of $R$ and 
$\Aut(R)$ is the group of ring automorphisms of $R$. Let $r^{\sigma(g)}$ denote the result of the action of $\sigma(g)$ on $r$. Then, for all $r\in R$ and $g,g_1,g_2,g_3\in G$,
the maps $\sigma$ and $\tau$ satisfy
\begin{equation*}
 \sigma(e)=1,\qquad\tau(e,g)=\tau(g,e)=1,
\end{equation*}
and
\begin{equation*}
 \tau(g_1,g_2)\tau(g_1g_2,g_3)=\tau(g_2,g_3)^{\sigma(g_1)}\tau(g_1,g_2g_3),\qquad
 r^{\sigma(g_2)\sigma(g_1)}=\tau(g_1,g_2)r^{\sigma(g_1g_2)}\tau(g_1,g_2)^{-1}.
\end{equation*}
These conditions guaranties that the product
\begin{equation*}
\Bigl(\sum_{g\in G}r_g\bar{g}\Bigr)\cdot \Bigl(\sum_{g\in G}s_g\bar{g}\Bigr):=\sum_{g\in G}\Bigl(\sum_{h_1h_2=g} r_{h_1}s_{h_2}^{\sigma(h_1)}\tau(h_1,h_2)\Bigr)\bar{g}
\end{equation*}
is associative.

Certain crossed products have their own specific name.
If the maps $\sigma$ and $\tau$ are trivial, the crossed product $R\ast G$ is the group ring $R[G]$.
If $\sigma$ is trivial, then $R\ast G=R^t[G]$ is a \emph{twisted group ring}, while if $\tau$ is trivial then $R\ast G=RG$ is a \emph{skew group ring}.

Given a 
normal subgroup $N\unlhd G$ and a fixed crossed product $R\ast G$, we have
\begin{equation}\label{crpro1}
R\ast G=\bigl( R\ast N\bigr)\ast G/N, 
\end{equation}
where the latter is 
some crossed product of the group $G/N$ over the ring $R\ast N$, and $R\ast N$ is the subring of $R\ast G$ induced by the subgroup $N$ \cite[Lemma 1.3]{Pas} 
(that is, the  maps $\sigma$ and $\tau$ associated to the crossed product $R\ast N$ are the restrictions of the
ones associated to $R\ast G$).
In particular,
\begin{equation}\label{crpro2}
R[G]=R[N]\ast G/N. 
\end{equation}
\smallskip

\noindent
We now state some interesting corollaries of our main result. 

There exist one-relator groups which are not residually finite~\cite{BaSo62}, or not even residually solvable~\cite{Ba69}. 
Whether or not all one-relator groups are sofic is a well-known open problem.

Wise recently proved that one-relator groups with torsion are residually finite~\cite{Wis}, answering a longstanding conjecture of Baumslag~\cite{Ba67}. 
Combining this deep result of Wise with our Main Theorem, we obtain:

\begin{corollary}\label{1-by-sofic}
Let $D$ be a division ring and $G$ be a \{finitely generated one-relator\}-by-sofic group. 
Then $D[G]$ is directly finite and, equivalently, stably finite.
\begin{proof}
Let $N\unlhd G$ be a normal subgroup of $G$ such that $G/N$ is sofic and $N$ is a finitely generated one-relator group.
If $N$ is torsion-free, then its group ring $D[N]$ embeds into a division ring $D'$ \cite{LeLe}.
By \eqref{crpro2} we have that $D[G]=D[N]\ast G/N$. Hence $D[G]$ embeds into $D'\ast G/N$, which is directly finite \cite[Corollary 5.4]{Vir}. Thus also $D[G]$ is directly finite.

If $N$ has torsion, then it is residually finite \cite{Wis}. Thus $G$ satisfies the hypotheses of the Main Theorem, and $D[G]$ is directly finite.
\end{proof}
\end{corollary}

From Corollary \ref{1-by-sofic}, in the particular case when the sofic group is trivial, we recover the fact that the group ring of a finitely generated one-relator group is directly and stably finite.

\begin{corollary}
Let $D$ be a division ring and $G$ be a finitely generated one-relator group, then $D[G]$ is directly finite and, equivalently, stably finite. 
\end{corollary}

Finitely generated right-angled Artin groups are known to be residually finite. We deduce the following:
\begin{corollary}
Let $N$ be a noetherian ring and $G$ be a \{finitely generated right-angled Artin group\}-by-sofic group. Then $N[G]$ is directly finite and, equivalently, stably finite.
\end{corollary}

\begin{remark}
Here is an alternative way of proving the Main Theorem, which was suggested by Simone Virili after the first version of this paper was written. 
In~\cite{Mon} it is observed that, given a noetherian ring $N$ and
$$\mathcal{C}:=\{\text{groups }G\text{ such that }N[G]\text{ is directly finite}\},$$
if a group $G$ is fully residually $\mathcal{C}$ then $G\in\mathcal{C}$. 
One then argues that a \{finitely generated residually finite\}-by-sofic group is residually finite-by-sofic, which is equivalent of being fully residually finite-by-sofic.
Hence the Main Theorem follows.

\end{remark}

\section{Examples}\label{sec.4}

We now apply our Main Theorem to some concrete groups, whose (non-)soficity still intrigues the experts. We conclude that they satisfy Kaplansky's direct and stable finiteness conjectures.
Moreover, we provide countably many new explicit presentations of groups satisfying these two conjectures.
These groups are not Locally Embeddable into Amenable (LEA for short), and it is not known whether or not they are sofic.

A finitely generated LEA group is the limit in $\mathcal{G}$ of amenable groups. In particular, LEA groups are sofic.
There exist examples of sofic groups which are not LEA, and the class of LEA groups is the biggest known class of groups strictly contained in the class of sofic groups.

In what follows, we use the fact that a finitely presented LEA group is residually amenable~\cite[Proposition 7.3.8]{CSC}. We refer to \cite[\S 7.3]{CSC} for more informations about LEA groups.

\subsection{Deligne's group}
A famous example considered by Deligne~\cite{De78} is the following.
Let $n\geq 2$ be an integer and $\widetilde{\Sp_{2n}(\Z)}$ be the preimage of the symplectic group $\Sp_{2n}(\Z)$ in the universal cover $\widetilde{\Sp_{2n}(\R)}$ of $\Sp_{2n}(\R)$.
It is known~\cite{De78} that $\widetilde{\Sp_{2n}(\Z)}$ is given by the following central extension
\begin{equation}\label{central.ext}
\{e\}\longrightarrow \Z\longrightarrow \widetilde{\Sp_{2n}(\Z)}\longrightarrow \Sp_{2n}(\Z)\longrightarrow\{e\}.
\end{equation}
The group $\widetilde{\Sp_{2n}(\Z)}$ is finitely presented as it is an extension of two finitely presented groups. 
Moreover, the group is not residually 
finite~\cite{De78} and it satisfies Kazhdan's property (T)~\cite[Example 1.7.13 (iii)]{BeHaVa}. This immediately implies that it is not an LEA group.

Our Main Theorem implies that $N[\widetilde{\Sp_{2n}(\Z)}]$ is directly finite for all noetherian rings $N$. Indeed, $\widetilde{\Sp_{2n}(\Z)}$ is \{finitely generated free\}-by-sofic, as shown by~\eqref{central.ext}.

\subsection{Finitely presented amalgamated products and HNN extensions}
From now on, if $\Gamma$ is a group then $\bar{\Gamma}$ denotes an isomorphic copy of $\Gamma$. If
$\Gamma=\langle X\mid R\rangle$ is a presentation of the group $\Gamma$, let $\bar{X}$ and $\bar{R}$ denote the same generators and relators in the isomorphic
copy $\bar\Gamma$.

Let $p>2$ be a prime number and $n\geq 3$. The group $\SL_n(\Z[1/p])$ is finitely presented~\cite[Theorem 4.3.21]{HO} and has Kazhdan's property (T)~\cite{BeHaVa}.
Moreover, it satisfies the \emph{congruence subgroup property}~\cite{BaMiSe}. 
This means that every finite index subgroup $H\leq \SL_n(\Z[1/p])$ contains the kernel of the natural projection
\begin{equation}\label{projection}
 \pi_q\colon\SL_n(\Z[1/p]) \twoheadrightarrow \SL_n(\Z/q\Z),
\end{equation}
for some $q$ coprime with $p$. In particular, if the finite index subgroup $H$ is normal in $\SL_n(\Z[1/p])$, it follows that
\begin{equation}\label{eq.quotients}
\frac{\SL_n(\Z[1/p])}{H}\cong\SL_n(\Z/q\Z) \qquad\text{or}\qquad\frac{\SL_n(\Z[1/p])}{H}\cong \PSL_n(\Z/q\Z),
\end{equation}
for exactly one $q$ coprime with $p$.

In what follows, given an element $x\in\SL_n(\Z[1/p])$ and a projection $\pi$ from $\SL_n(\Z[1/p])$ onto $\SL_n(\Z/q\Z)$ or $\PSL_n(\Z/q\Z)$, we denote the order of $\pi(x)$ by $o_x$.

The proof of the following theorem is adapted from \cite[Theorem 1]{KaNi}, where an analogous fact is proved, but in the case when the amalgamated subgroup is infinite cyclic.
In that case, the resulting group is known to be sofic. In contrast to \cite{KaNi}, our aim is to produce non-LEA groups that are not known to be sofic and that satisfy Kaplansky's direct and 
stable finiteness conjectures.
\begin{theorem}\label{theoKaNi}
Let $p>2$ be a prime number, $n\geq 3$,
let $\Gamma:=\SL_n(\Z[1/p])=\langle X\mid R\rangle$. Let $\langle a,b\rangle =F\leq\Gamma$ be the subgroup generated by the matrices
\begin{equation*} a=\begin{pmatrix}
     1&2&0\\ 0&1&0\\ 0&0& \I_{n-2}
     \end{pmatrix}, 
  \qquad b=\begin{pmatrix}
     1&0&0\\ 2&1&0\\ 0&0& \I_{n-2}~
     \end{pmatrix},
\end{equation*}
where $\I_{n-2}$ is the identity matrix of dimension $n-2$. Then the group $$G:=\Gamma\ast_F \Gamma=\langle X,\bar{X}\mid R,\bar{R},a=\bar{a},b=\bar{b}\rangle$$ is not LEA.
\begin{proof}
The group $G$ is finitely presented. Hence it is sufficient to prove that it is not residually amenable.
Let 
$$ x=\begin{pmatrix}
     1&\frac{2}{p}&0\\ 0&1&0\\ 0&0&\I_{n-2}
     \end{pmatrix}
$$
and consider the element $g=[x,\bar{x}]\in G$.
Since $x\notin F$, using normal forms for the elements of the amalgamated free product \cite[I.11]{LS}, it follows that $g\neq e_G$. 

Let $\pi\colon G\twoheadrightarrow A$ be a surjective homomorphism with $A$ amenable. We claim that $\pi(g)=e_A$. 
Indeed, consider the restriction $\pi\restriction_\Gamma\colon\Gamma\to \pi(\Gamma)$. The group $\pi(\Gamma)\leq A$ is amenable and moreover it is a quotient
of $\Gamma$, which is a group with Kazhdan's property $(T)$. 
Hence $\pi(\Gamma)$ is finite and, in particular, $\pi(x)$ has finite order $o_x$.

The element $x$ is unipotent, so $\pi(x)$ is unipotent too.
As the group $\Gamma$ satisfies the congruence subgroup property, it follows that $\pi(x)$ is an element of some $\SL_n(\Z/q\Z)$ or $\PSL_n(\Z/q\Z)$, for $q$ coprime with $p$.
As $\pi(x)$ is unipotent, the order $o_x$ divides a power of $q$. Moreover $\gcd(p,q)=1$, so $\gcd(p,o_x)=1$.

As $x^p=a$, we have that $\langle \pi(a)\rangle\leq\langle \pi(x)\rangle$ and that
$$o_a=o_{x^p}=\frac{o_x}{\gcd(p, o_x)}=o_x.$$
This implies that the two finite groups $\langle \pi(a)\rangle$ and $\langle \pi(x)\rangle$ have the same cardinality, and so 
$\langle \pi(a)\rangle=\langle \pi(x)\rangle$.

The same argument applies to the elements $\bar{x}$ and $\bar{a}$, so $\langle \pi(\bar{a})\rangle=\langle \pi(\bar{x})\rangle$.
As in the group $G$ we have $a=\bar{a}$, it follows that $\langle\pi(x)\rangle=\langle\pi(\bar{x})\rangle$, and so $\pi(g)=[\pi(x),\pi(\bar{x})]=e_A$. That is, the element $g$ is mapped 
to the trivial element in all amenable quotients of $G$. Thus, $G$ is not residually amenable.
\end{proof}
\end{theorem}

In the next corollary we construct countably many pairwise non-isomorphic groups, not known to be sofic, satisfying Kaplansky's direct and stable finiteness conjectures.

\begin{corollary}\label{corKaNi}
With the notations of the previous theorem, for $r\geq 2$ let $F_r\leq\Gamma$ be generated by
$\{b^iab^{-i}\mid i=0,\dots,r-1\}$. Then the groups
\begin{equation*}
 G_r:=\Gamma\ast_{F_r}\Gamma=\langle X,\bar{X}\mid R,\bar R,a=\bar{a},\dots ,b^{r-1}ab^{-r+1}=\bar b^{r-1}\bar a\bar b^{-r+1}\rangle
\end{equation*}
are pairwise non-isomorphic and are not LEA. Moreover they are free-by-sofic, and hence $D[G_r]$ is directly finite (equivalently: stably finite) with respect to all division rings $D$, for all $r\geq2$.
\begin{proof}
The subgroup $F_r$ is a free group of rank $r$. The argument of the proof of Theorem~\ref{theoKaNi}
shows that the element $g=[x,\bar{x}]$ is mapped to the trivial element in all amenable quotients of $G_r$. Hence $G_r$ is not LEA.

By the universal property of amalgamated free products, we have the following commuting diagram
$$\xymatrix{
\Gamma \,\,\ar@{^{(}->}[r]\ar[ddr]_{\id} &G_r\ar@{.>}^{\exists ! \varphi}[dd] & \,\,\bar\Gamma\ar@{_{(}->}[l]\ar[ddl]^{\id}\\
&&\\
&\Gamma&
}$$
where $\varphi\colon G_r\twoheadrightarrow \Gamma$ is a surjective homomorphism.
Let $K=\ker\varphi$, then $K\cap \Gamma=K\cap \bar{\Gamma}=\{e\}$. This implies that $K$ acts freely on the Bass-Serre tree associated to
the amalgamated free product $G_r$, that is to say, $K$ is a free group. 

Hence $G_r$ is free-by-sofic and, given a division ring $D$, the group ring $D[G_r]$ is directly finite by our Main Theorem.
Note that $K$ is not finitely generated, so we cannot conclude that $N[G_r]$ is directly finite for noetherian rings~$N$.

It remains to prove that the family $\{G_r\}_{r\geq 2}$ consists of pairwise non-isomorphic finitely presented groups. To this aim, we recall the notion of deficiency of a finitely presented group.

The \emph{deficiency} $\df(G)$ of a finitely presented group $G$ is defined as
$\max\{\lvert X\rvert-\lvert R\rvert \}$ over all the finite presentations $G=\langle X\mid R\rangle$. It is invariant under isomorphism, and we have
$$\df(G_r)=2\cdot \df(\SL_m(\Z[1/p])-r.$$
Therefore, the groups $\{G_r\}_{r\geq 2}$ are pairwise non-isomorphic.

\end{proof}
\end{corollary}

Note that the groups $G_r$ do not have property (T). This follows, for instance, from \cite[Remark 2.3.5 and Theorem 2.3.6]{BeHaVa}.

Our result extends further to HNN extensions. In \cite{Ber}, we have characterized the residual amenability of particular HNN extensions $A\ast_H$ and amalgamated free products $A\ast_H A$, in terms of the
amalgamated subgroup $H$ being closed in the proamenable topology of $A$ \cite[Corollaries 1.8 and 1.10]{Ber}.
Using these results and Corollary~\ref{corKaNi}, we obtain:

\begin{corollary}\label{cor2}
Let $p>2$ be a prime number, $n\geq 3$ and $\SL_n(\Z[1/p])=\langle X\mid R\rangle$. For $r\geq 2$ the groups
\begin{equation*}
\Gamma_r:=\langle X, t\mid R,\, tat^{-1}=a,\, t(bab^{-1})t^{-1}=bab^{-1},\dots,\,t(b^{r-1}ab^{-(r-1)})t^{-1}=b^{r-1}ab^{-(r-1)}\rangle 
\end{equation*}
are pairwise non-isomorphic and are not LEA. Moreover they are free-by-sofic, and hence $D[\Gamma_r]$ is directly finite (equivalently: stably finite) with respect to all division rings $D$, for all $r\geq 2$.
\end{corollary}

We end with the following question:
\begin{question}
Are the groups $G_r$ and $\Gamma_r$ sofic/hyperlinear? 
\end{question}

\end{document}